\documentclass[12pt]{amsart}
\usepackage{graphicx}
\usepackage[headings]{fullpage}
\usepackage{amssymb,epic,eepic,epsfig,amsbsy,amsmath,amscd}
\numberwithin{equation}{section}
                        \theoremstyle{plain}
\usepackage{mathrsfs, enumerate}

\newcommand{\psdraw}[2]
         {\begin{array}{c} \hspace{-1.3mm}
         \raisebox{-4pt}{\psfig{figure=#1.eps,width=#2}}
         \hspace{-1.9mm}\end{array}}
         
\newcommand\no[1]{}

\newtheorem{theorem}{Theorem}[section]

\newtheorem{lemma}[theorem]{Lemma}
\newtheorem{corollary}[theorem]{Corollary}

\theoremstyle{definition}
\newtheorem{remark}[theorem]{Remark}

\newcommand{\C}{\Bbb C}

\newcommand{\tr}{{\mathrm{tr}\,}}
\newcommand{\la}{\langle}
\newcommand{\ra}{\rangle}

\def\BC{\mathbb C}
\def\C{\mathbb C}

\def\BZ{\mathbb Z}

\def\CL{\mathcal L}

\def\la{\langle}
\def\ra{\rangle}

\def\be { \begin{equation} }
\def\ee { \end{equation} }

\begin{document}

\title[On the volume of double twist link cone-manifolds]
{On the volume of double twist link cone-manifolds}

\author{Anh T. Tran}

\begin{abstract} 
We consider the double twist link $J(2m+1, 2n+1)$ which is the two-bridge link corresponding to the continued fraction $(2m+1)-1/(2n+1)$.
It is known that $J(2m+1, 2n+1)$ has reducible nonabelian $SL_2(\BC)$-character variety if and only if $m=n$. 
In this paper we give a formula for the volume of hyperbolic cone-manifolds of $J(2m+1,2m+1)$. We also give a formula for the A-polynomial 2-tuple 
corresponding to the canonical component of the character variety of $J(2m+1,2m+1)$.
\end{abstract}



\address{Department of Mathematical Sciences, The University of Texas at Dallas, 
Richardson, TX 75080, USA}
\email{att140830@utdallas.edu}

\maketitle

\section{Introduction}\label{sec:intro}

For a hyperbolic link $\CL$ in $S^3$, let $E_\CL = S^3 \setminus \CL$ be the link exterior and let $\rho_{\text{hol}}$ be a holonomy representation of $\pi_1(E_\CL)$ into $PSL_2(\BC)$. Thurston \cite{Th} showed that $\rho_{\text{hol}}$ can be deformed into an $\ell$-parameter family $\{\rho_{\alpha_1, \cdots, \alpha_\ell}\}$ of representations to give a corresponding family $\{E_\CL(\alpha_1, \cdots, \alpha_\ell)\}$ of singular complete hyperbolic manifolds, where $\ell$ is the number of components of $\CL$. In this paper we  consider only the case where all of $\alpha_j$'s are equal to a single parameter $\alpha$. In which case we also denote $E_\CL(\alpha_1, \cdots, \alpha_\ell)$ by $E_\CL(\alpha)$. These $\alpha$'s and $E_\CL(\alpha)$'s are called the cone-angles and hyperbolic cone-manifolds of $\CL$, respectively. 
We consider the complete hyperbolic structure on a link complement as the cone-manifold structure with cone-angle zero. It is known that for a two-bridge  link $\CL$ there exists an angle $\alpha_\CL \in [\frac{2\pi}{3},\pi)$ such that $E_\CL(\alpha)$ is hyperbolic for $\alpha \in (0, \alpha_\CL)$, Euclidean for $\alpha =\alpha_\CL$, and spherical for $\alpha \in (\alpha_\CL, \pi)$  \cite{HLM, Ko, Po, PW}. A method for computing the volume of hyperbolic cone-manifolds of links was outlined in \cite{HLM}, and explicit volume formulas have been known for hyperbolic cone-manifolds of the links $5_1^2, \, 6_2^2, \, 6_3^2, \, 7_3^2$ (see \cite{HLMR-link} and references therein) and of twisted Whitehead links \cite{Tr-Whitehead}. 

For integers $m$ and $n$, consider the double twist link $J(2m+1,2n+1)$ 
which is the two-bridge link corresponding to the continued fraction $(2m+1)-1/(2n+1)$ (see Figure 1). 
It was shown by Petersen and the author \cite{PT} that $J(2m+1, 2n+1)$ has reducible nonabelian $SL_2(\BC)$-character variety if and only if $m=n$. 
In this paper we are interested in the double twist link $\CL_m = J(2m+1,2m+1)$, since the canonical component of the character variety of $\CL_m$ has a rather nice form (see Remark \ref{fact}). Here a canonical component of the character variety of a hyperbolic link $\CL$ is a component containing the character of a lift of a holonomy representation of $\pi_1(E_\CL)$ to $SL_2(\BC)$. 

\begin{figure}[htpb] \label{Wh}
$$\psdraw{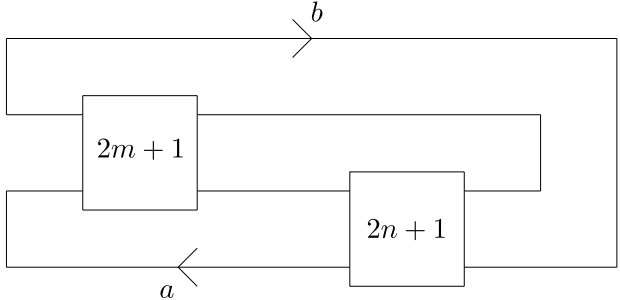}{4in}$$
\caption{The double twist link $J(2m+1,2n+1)$. Here $2m+1$ and $2n+1$ denote the numbers of half twists
in the boxes. Positive (resp. negative) numbers correspond to right-handed (resp. left handed)
twists.}
\end{figure}

Let $\{S_j(v)\}_{j \in \BZ}$ be the sequence of Chebychev polynomials of the second kind defined by 
$S_0(v)=1$, $S_1(v)=v$ and $S_{j}(v) = v S_{j-1}(v) - S_{j-2}(v)$ 
for all integers $j$. Let 
$$
R_{\CL_m}(s,z) = (s^2+s^{-2}+2-z) \big( S^2_m(z) + S^2_{m-1}(z) \big) - 2(s^2+s^{-2}) S_m(z) S_{m-1}(z) -z.
$$
The volume of the hyperbolic cone-manifold of $\CL_m$ is computed as follows. 

\begin{theorem} \label{thm:volume}
For $\alpha \in (0, \alpha_{\CL_m})$ we have
$$\emph{Vol} \, E_{\CL_m}(\alpha) = \int_{\alpha}^{\pi} \log \left| \frac{S_m(z) - e^{-i\omega} S_{m-1}(z)}{S_m(z) - e^{i\omega} S_{m-1}(z)} \right| d\omega$$ 
where $z$, with $\emph{Im}(S_{m-1}(z) \overline{S_m(z)}) \ge 0$, is a certain root of $R_{\CL_m}(e^{i\omega/2},z) =0$.
\end{theorem}

Note that the above volume formula for the hyperbolic cone-manifold $E_{\CL_m}(\alpha)$ 
depends on the choice of a root $z$, with $\text{Im}(S_{m-1}(z) \overline{S_m(z)}) \ge 0$, of  $R_{\CL_m}(e^{i\omega/2},z) =0$. 
In numerical approximations, we choose the root $z$ which gives the maximal volume. 

It is known that the volume of the $k$-fold cyclic covering over a hyperbolic link $\CL$ 
is $k$ times the volume of the hyperbolic cone-manifold of $\CL$ with cone-angle $2\pi /k$.
As a direct consequence of Theorem \ref{thm:volume}, we obtain the following. 

\begin{corollary}
The hyperbolic volume of the $k$-fold cyclic covering over 
the two-bridge link $\CL_m$, with $k \ge 3$, 
is given by the following formula
$$k \, \emph{Vol} \, E_{\CL_m}(\frac{2\pi}{k}) = k \int_{\frac{2\pi}{k}}^{\pi} \log \left| \frac{S_m(z) - e^{-i\omega} S_{m-1}(z)}{S_m(z) - e^{i\omega} S_{m-1}(z)} \right| d\omega$$ 
where $z$, with $\emph{Im} \left( S_{m-1}(z)S_m(z) \right) \ge 0$, is a certain root of $R_{\CL_m}(e^{i\omega/2},z) =0$.
\end{corollary} 

The A-polynomial of a knot in $S^3$ was introduced by Cooper, Culler, Gillet, Long and Shalen \cite{CCGLS} in the 90's. 
It describes the $SL_2(\BC)$-character variety of the knot complement as viewed from the boundary torus. 
The A-polynomial carries a lot of information about the topology of the knot.
 For example, the sides of the Newton polygon of the A-polynomial of a knot in $S^3$ give rise to incompressible surfaces in the knot complement \cite{CCGLS}. 
 A generalization of the A-polynomial to links in $S^3$ was proposed by Zhang \cite{Zh}. For an $\ell$-component link in $S^3$, 
 Zhang defined a polynomial $\ell$-tuple link invariant called the A-polynomial $\ell$-tuple. 
 The A-polynomial 1-tuple of a knot is just its A-polynomial.
 The A-polynomial $\ell$-tuple also caries important information about the topology of the link. 
 For example, it can be used to construct concrete examples of hyperbolic link manifolds with non-integral traces \cite{Zh}.
 
 The A-polynomial $2$-tuple  has been computed for a family of two bridge links called twisted Whitehead links \cite{Tr-Whitehead}.
 In this paper we compute the A-polynomial 2-tuple for 
 the canonical component of the character variety of $\CL_m = J(2m+1,2m+1)$. 
  
\begin{theorem} \label{thm:A-polynomial}
Let $\{Q_j(s,w)\}_{j \in \BZ}$ be the sequence of polynomials in two variables $s,w$ defined by $Q_{-1} = Q_0 = 2$ and
$$Q_j = \alpha Q_{j-1} - Q_{j-2} +\beta$$ where
\begin{eqnarray*}
\alpha &=& (s^8+s^4) w^4+(-2 s^8+6 s^6+6 s^4-2 s^2) w^3 + (s^8-12 s^6+34 s^4-12 s^2+1) w^2\\
&& \qquad \qquad \quad  + \, (-2 s^6+6 s^4+6 s^2-2) w+s^4+1,\\
\beta &=& -2(s^2-1)^2 \left( s^4 w^4 - (s^4+s^2) w^3 - 6 s^2 w^2 - (s^2+1)w+1\right).
\end{eqnarray*}
Then the A-polynomial 2-tuple corresponding to the canonical component of the character variety of $\CL_m$ is  
$[A(M,L), A(M,L)]$ where $A(M,L) = (L-1)Q_m(M, LM^{2m})$.
\end{theorem}

The paper is organized as follows. In Section \ref{sec:prelim} we review the definition of the A-polynomial $\ell$-tuple of an $\ell$-component link in $S^3$. In Section \ref{sec:DTL} we compute the nonabelian $SL_2(\BC)$-representations of the double twist link $J(2m+1,2n+1)$. In Section \ref{sec:volume} we compute the volume of hyperbolic cone-manifolds of $\CL_m = J(2m+1,2m+1)$ and give a proof of Theorem \ref{thm:volume}. The last section is devoted to the computation of the A-polynomial 2-tuple for the canonical component of the character variety of $\CL_m$ and a proof of Theorem \ref{thm:A-polynomial}.

\section{The A-polynomial $\ell$-tuple of a link} \label{sec:prelim}

\subsubsection{Character varieties}

The set of characters of representations of a finitely generated group $G$ into $SL_2(\BC)$ is known to be a algebraic set over $\BC$ \cite{CS, LM}. It is called the character variety of $G$ and denoted by $\chi(G)$. For example, the character variety $\chi(\BZ^2)$ of the free abelian group on 2 generators $\mu,\lambda$ is isomorphic to $(\BC^*)^2/\tau$, where $(\BC^*)^2$ is the
set of non-zero complex pairs $(M,L)$ and $\tau: (\BC^*)^2 \to (\BC^*)^2$ is the involution
defined by $\tau(M,L)=(M^{-1},L^{-1})$. This fact can be proved by noting that every representation $\rho: \BZ^2 \to SL_2(\BC)$ is
conjugate to an upper diagonal one, with $M$ and $L$ being the
upper left entries of $\rho(\mu)$ and $\rho(\lambda)$ respectively. 

\subsubsection{The A-polynomial}

Suppose $\CL = K_1 \sqcup \dots \sqcup K_\ell$ be an $\ell$-component link in $S^3$. 
Let $E_\CL = S^3 \setminus \CL$ be the link exterior and $T_1, \dots, T_\ell$ 
the boundary tori of $E_\CL$ corresponding to $K_1, \dots, K_\ell$ respectively. 
Each $T_j$ is a torus whose fundamental group  is free abelian of rank
two. An orientation of $K_j$ will define a unique pair of an
oriented meridian $\mu_j$ and an oriented longitude $\lambda_j$ such that the linking
number between the longitude $\lambda_j$ and the knot $K_j$ is 0. The pair provides
an identification of $\chi(\pi_1(T_j))$ and $(\BC^*)_j^2/\tau_j$, where $(\BC^*)_j^2$ is the
set of non-zero complex pairs $(M_j,L_j)$ and $\tau_j$ is the involution
$\tau(M_j,L_j)=(M_j^{-1},L_j^{-1})$, which actually does not depend on the orientation of $K_j$.

The inclusion $T_j \hookrightarrow E_{\CL}$ induces the restriction
map
$$\rho_j : \chi(\pi_1(E_{\CL})) \longrightarrow \chi(\pi_1(T_j))\equiv (\BC^*)_j^2/\tau_j.$$
For each $\gamma \in \pi_1(E_{\CL})$ let $f_{\gamma}$ be the regular function on $\chi(\pi_1(E_{\CL}))$ defined by $$f_{\gamma}(\chi_{\rho})=(\chi_{\rho}(\gamma))^2-4 = (\tr \rho(\gamma))^2-4,$$ where $\chi_\rho$ denotes the character of a representation 
$\rho: \pi_1(E_{\CL}) \to SL_2(\BC)$. Let $\chi_j(\pi_1(E_{\CL}))$ be the subvariety of $\chi(\pi_1(E_{\CL}))$ defined by $f_{\mu_k}=0,~f_{\lambda_k}=0$ for all $k \not= j.$
Let $Z_j$ be the image of $\chi_j(\pi_1(E_{\CL}))$ under 
$\rho_j$ and  $\hat Z_j \subset (\BC^*)_j^2$ the lift of $Z_j$ under the
projection $(\BC^*)_j^2 \to (\BC^*)_j^2/\tau_i$. It is known that the Zariski closure of
$\hat Z_j\subset (\BC^*)_j^2 \subset \BC_j^2$ in $\BC_j^2$ is an algebraic set
consisting of components of dimension 0 or 1 \cite{Zh}. The union of all the
1-dimension components is defined by a single polynomial $A_j \in
\BZ[M_j,L_j]$ whose coefficients are co-prime. Note that $A_j$ is defined up to $\pm 1$. 
We will call $[A_1(M_1, L_1), \cdots, A_\ell(M_\ell, L_\ell)]$ the A-polynomial $\ell$-tuple of $\CL$. 
For brevity, we also write $A_j(M,L)$ for $A_j(M_j, L_j)$. We refer the reader to \cite{Zh} for properties of the A-polynomial $\ell$-tuple.


\section{Double twist links $J(2m+1,2n+1)$} \label{sec:DTL}

In this section  we compute nonabelian $SL_2(\BC)$-representations of the double twist link $J(2m+1,2n+1)$. They are described by the Chebyshev polynomials of the second kind, and so we first recall some properties of these polynomials.

\subsection{Chebyshev polynomials} Recall that $\{S_j(v)\}_{j \in \BZ}$ is the sequence of the Chebychev polynomials of the second kind defined by 
$S_0(v)=1$, $S_1(v)=v$ and $S_{j}(v) = v S_{j-1}(v) - S_{j-2}(v)$ 
for all integers $j$. The following two lemmas are elementary, see e.g. \cite{Tr-Whitehead}.

\begin{lemma} \label{chev1}
For any integer $j$ we have
$$S^2_j(v) + S^2_{j-1}(v) - v S_j(v) S_{j-1}(v) =1.$$
\end{lemma}

\begin{lemma} \label{chev2}
Suppose $V \in SL_2(\BC)$ and $v=\tr V$. For any integer $j$ we have
$$V^j = S_{j}(v) \mathbf{1} - S_{j-1}(v) V^{-1}$$
where $\mathbf{1}$ denotes the $2 \times 2$ identity matrix.
\end{lemma}

We will need the following lemma in the last section of the paper.

\begin{lemma} \label{chev3}
For any integer $j$ we have
$$S_j(z) S_{j-1}(z)=(z^2-2)S_{j-1}(z) S_{j-2}(z)-S_{j-2}(z) S_{j-3}(z)+z.$$
\end{lemma}

\begin{proof}
We have $S_j(z) S_{j-1}(z) + S_{j-2}(z) S_{j-3}(z)$
\begin{eqnarray*}
&=& (z S_{j-1}(z) - S_{j-2}(z)) S_{j-1}(z) + S_{j-2}(z) (z S_{j-2}(z) - S_{j-1}(z)) \\
&=& z (S^2_{j-1}(z) + S^2_{j-2}(z)) - 2 S_{j-1}(z) S_{j-2}(z).
\end{eqnarray*}
The lemma follows, since $S^2_{j-1}(z) + S^2_{j-2}(z) = 1 + z S_{j-1}(z) S_{j-2}(z)$ by Lemma \ref{chev1}.
\end{proof}

\subsection{Nonabelian reprsentations} \label{nab}

In this subsection we study representations of link groups into $SL_2(\BC)$. 
A representation is called nonabelian if its image is a nonabelian subgroup of $SL_2(\BC)$. 
Let $\CL = J(2m+1,2n+1)$ 
and $E_\CL = S^3 \setminus \CL$ the link exterior. By \cite{PT} (and \cite{MPL} also) the link group of $\CL$ has a two-generator presentation
$$\pi_1(E_\CL) = \la a, b \mid a w = wa \ra,$$
where  $w=(b^{-1}a)^m \big[ (ba^{-1})^m ba (b^{-1}a)^m \big]^n$ and $a,b$ are meridians depicted in Figure 1.

Suppose $\rho: \pi_1(E_\CL) \to SL_2(\C)$ is a nonabelian representation. 
Up to conjugation, we may assume that 
\begin{equation} \label{rep}
\rho(a) = \left[ \begin{array}{cc}
s_1 & 1 \\
0 & s_1^{-1} \end{array} \right] \quad \text{and} \quad 
\rho(b) = \left[ \begin{array}{cc}
s_2 & 0 \\
u & s_2^{-1} \end{array} \right]
\end{equation}
where $(u, s_1, s_2) \in (\C^*)^3$ satisfies the matrix equation $\rho(aw) = \rho(wa)$. 
For any word $v$ in 2 letters $a$ and $b$, we write $\rho(v) = \left[ \begin{array}{cc}
v_{11} & v_{12} \\
v_{21} & v_{22} \end{array} \right]$.  Then, by Riley \cite{Ri}, $w_{12}$ can be written as $w_{21} = u w'_{21}$ for some $w'_{12} \in \BC[s_1^{\pm 1},s_2^{\pm 1}, u]$ and the matrix equation $\rho(aw) = \rho(wa)$ is equivalent to the single equation $w'_{12}=0$. We call $w'_{12}$ the Riley polynomial of $\CL$.

We now compute $w'_{12}$ explicitly. Let $x = \tr \rho(a) = s_1 + s_1^{-1}$, $y = \tr \rho(b) = s_2 +s_2^{-1}$ and $z= \tr \rho(ab^{-1}) = s_1 s_2^{-1} + s_1^{-1} s_2 -u$. 

Let $c=(b^{-1}a)^m$ and $d = (ba^{-1})^m ba (b^{-1}a)^m = bc^{-1} a c$. Then $w=c d^n$. Since 
$$\rho(b^{-1}a) = \left[
\begin{array}{cc}
s_1 s_2^{-1} & s_2^{-1} \\
-s_1 u & s_1^{-1} s_2 - u \\
\end{array}
\right],$$ by Lemma \ref{chev2}
we have $\rho(c) = \left[
\begin{array}{cc}
c_{11} & c_{12} \\
c_{21} & c_{22} \\
\end{array}
\right]$ where 
\begin{eqnarray*}
c_{11} &=& S_m(z) - (s_1^{-1} s_2 - u) S_{m-1}(z), \\
c_{12} &=& s_2^{-1}S_{m-1}(z),\\
c_{21} &=& -s_1 u S_{m-1}(z), \\ 
c_{22} &=& S_m(z) - s_1 s_2^{-1} S_{m-1}(z).
\end{eqnarray*}
By a direct computation we then have $\rho(d) = \rho(bc^{-1} a c) = \left[
\begin{array}{cc}
d_{11} & d_{12} \\
d_{21} & d_{22} \\
\end{array}
\right]$
where
\begin{eqnarray*}
d_{11} &=& s_1 s_2 S^2_m(z) - (s_1^2 + s_2^2) S_m(z) S_{m-1}(z) + (s_1 s_2 + u) S^2_{m-1}(z),\\
d_{12} &=& s_2 S^2_m(z) - (s_1 + s_1^{-1}) S_m(z) S_{m-1}(z) + s_2^{-1} S^2_{m-1}(z),\\
d_{21} &=& u \big( s_1 S^2_m(z) - (s_2 + s_2^{-1}) S_m(z) S_{m-1}(z) + s_1^{-1} S^2_{m-1}(z) \big),\\
d_{22} &=& (s_1^{-1} s_2^{-1} + u) S^2_m(z) - (s_1^{-2} + s_2^{-2}) S_m(z) S_{m-1}(z) +  s_1^{-1} s_2^{-1} S^2_{m-1}(z).
\end{eqnarray*}
Let $t = \tr \rho(d)$. From the above computations we have 
\begin{eqnarray*}
t &=& (s_1 s_2 + s_1^{-1} s_2^{-1} + u) \big( S^2_m(z) + S^2_{m-1}(z) \big) - (s_1^2 + s_1^{-2} + s_2^2  + s_2^{-2}) S_m(z) S_{m-1}(z) \\
&=& (xy-z) \big( S^2_m(z) + S^2_{m-1}(z) \big) - (x^2+y^2-4) S_m(z) S_{m-1}(z).
\end{eqnarray*}

Since $w=c d^n$, by Lemma \ref{chev2} we have 
$$\rho(w) = \left[
\begin{array}{cc}
c_{11} & c_{12} \\
c_{21} & c_{22} \\
\end{array}
\right] \left[
\begin{array}{cc}
S_n(t) - d_{22} S_{n-1}(t)  & d_{12} S_{n-1}(t) \\
d_{21} S_{n-1}(t) & S_n(t) - d_{11} S_{n-1}(t) \\
\end{array}
\right].$$
With $\rho(w) = \left[ \begin{array}{cc}
w_{11} & w_{12} \\
w_{21} & w_{22} \end{array} \right]$ we obtain
\begin{eqnarray*}
w_{11} &=& c_{11} \big( S_n(t) - d_{22} S_{n-1}(t) \big) + c_{12} d_{21}  S_{n-1}(t),\\
w_{21} &=& c_{21} \big( S_n(t) - d_{22} S_{n-1}(t) \big) + c_{22} d_{21} S_{n-1}(t) .
\end{eqnarray*}

By direct computations we have $w_{21} = u s_1 \big(S_m(z) S_{n-1}(t) - S_{m-1}(z) S_n(t) \big)$ and
\begin{eqnarray*}
w_{11} &=& - S_{n-1}(t) \big\{ (s_1 s_2^{-1} + s_1^{-1} s_2 + s_1^{-1} s_2^{-1} -z) S_m(z) - s_1^{-2} S_{m-1}(z) \big\} \\
       && + \, S_n(t) \big( S_m(z) + (s_1 s_2^{-1} - z) S_{m-1}(z) \big).        
\end{eqnarray*}
Hence, the Riley polynomial of $\CL = J(2m+1,2n+1)$ is $$w'_{21} = S_m(z) S_{n-1}(t) - S_{m-1}(z) S_n(t).$$ It determines the nonabelian $SL_2(\BC)$-character variety of $\CL$, which is essentially the set of all nonabelian representations $\rho: \pi_1(E_\CL) \to SL_2(\BC)$ up to conjugation. Moreover, for any nonabelian representation $\rho$ of the form \eqref{rep} we have 
$\rho(w) = \left[ \begin{array}{cc}
w_{11} & * \\
0 & (w_{11})^{-1} \end{array} \right]$ where 
\begin{equation} \label{w11}
w_{11} = - S_{n-1}(t) \big\{ (s^{-1}_1 s_2  + s_1^{-1} s_2^{-1} ) S_m(z) - s_1^{-2} S_{m-1}(z) \big\} +  S_n(t) S_m(z) .        
\end{equation}

Let $\overline{w}$ is the word obtained from $w$ by exchanging $a$ and $b$, namely $$\overline{w} = (a^{-1}b)^m \big[ (ab^{-1})^m ab (a^{-1}b)^m \big]^n.$$
It is easy to see that the equation $aw = wa$ is equivalent to $\overline{w} b = b \overline{w}$. Moreover, for any nonabelian representation $\rho$ of the form \eqref{rep} we have 
$\rho(\overline{w}) = \left[ \begin{array}{cc}
\overline{w}_{11} & 0\\
* & (\overline{w}_{11})^{-1} \end{array} \right]$ where 
\begin{equation} \label{ow11}
\overline{w}_{11} = - S_{n-1}(t) \big\{ (s_1 s_2^{-1} + s_1^{-1} s_2^{-1}) S_m(z) - s_2^{-2} S_{m-1}(z) \big\} +  S_n(t) S_m(z).        
\end{equation}

\begin{remark} \label{fact}
The above formula for the nonabelian $SL_2(\BC)$-character variety of the double twist link $\CL = J(2m+1, 2n+1)$
was already obtained in \cite{PT} by a different method. Moreover, it was also shown in \cite{PT} that the nonabelian character variety of $\CL$ is reducible if and only if $m=n$. In which case, it has exactly 2 irreducible components and the canonical component is determined by the equation $t=z$.  
\end{remark}

From now on we consider only the double twist link $\CL_m = J(2m+1, 2m+1)$, where $m \not= -1, 0$. As mentioned above, the canonical component of the character variety of $\CL_m$ is given by the equation $t=z$ where
\begin{equation} \label{t}
t = (xy-z) \big( S^2_m(z) + S^2_{m-1}(z) \big) - (x^2+y^2-4) S_m(z) S_{m-1}(z).        
\end{equation}

\section{Volume of hyperbolic cone-manifolds of $\CL_m$} \label{sec:volume}

Recall that $E_{\CL_m}(\alpha)$ is the cone-manifold of $\CL_m$ with cone angles $\alpha_1 = \alpha_2 = \alpha$. 
There exists an angle $\alpha_{\CL_m} \in [\frac{2\pi}{3},\pi)$ such that $E_{\CL_m}(\alpha)$ is hyperbolic for $\alpha \in (0, \alpha_{\CL_m})$, Euclidean for $\alpha =\alpha_{\CL_m}$, and spherical for $\alpha \in (\alpha_{\CL_m}, \pi)$.

For $\alpha \in (0, \alpha_{\CL_m})$, by the Schlafli formula we have 
$$
\text{Vol} \, E_{\CL_m}(\alpha) 
= \int_{\alpha}^{\pi} 2\log \left| w_{11} \right| d\omega 
$$
where $w_{11}$ is the $(1,1)$-entry of the matrix $\rho(w)$ and $\rho: \pi_1(\CL_m) \to SL_2(\BC)$ is a representation of the form \eqref{rep} 
such that the following 3 conditions hold:
\begin{enumerate}[(i)]
\item $s_1 = s_2 = s = e^{i\omega/2}$,
\item the character $\chi_\rho$ of $\rho$ lies on the canonical component of the character variety of $\CL_m$,
\item $|w_{11}| \ge 1$.
\end{enumerate}
We refer the reader to \cite{HLM, HLMR-link} and references therein for the volume formula of hyperbolic cone-manifolds of
links using the Schlafli formula.

We now simplify $w_{11}$ for representations $\rho$ of the form \eqref{rep} satisfying the conditions (i)--(iii). Consider the canonical component $t=z$ of the character variety of $\CL_m$. With $s_1 = s_2 = s = e^{i\omega/2}$, equation \eqref{w11} implies that
\begin{eqnarray*}
w_{11} &=& - S_{m-1}(z) \big\{ (1+s^{-2}) S_m(z) - s^{-2} S_{m-1}(z) \big\} + S^2_m(z) \\
&=& \big( S_m(z) - S_{m-1}(z) \big) \big( S_m(z) - s^{-2}S_{m-1}(z) \big).
\end{eqnarray*}
Moreover, the equation $t=z$ can be written as  $$(s^2+s^{-2}+2-z) \big( S^2_m(z) + S^2_{m-1}(z) \big) - 2(s^2+s^{-2}) S_m(z) S_{m-1}(z)=z.$$
This, together with $S^2_m(z) + S^2_{m-1}(z) = 1 + z S_m(z) S_{m-1}(z)$ (by Lemma \ref{chev2}), implies that 
\begin{eqnarray*}
S_m(z) S_{m-1}(z) &=& \frac{2z-(s^2+s^{-2}+2)}{(z-2)(s^2+s^{-2}-z)},\\
S^2_m(z) + S^2_{m-1}(z) &=& \frac{z^2-2(s^2+s^{-2})}{(z-2)(s^2+s^{-2}-z)}.
\end{eqnarray*}
Then $\big( S_m(z) - S_{m-1}(z) \big)^2 = S^2_m(z) + S^2_{m-1}(z) - 2 S_m(z) S_{m-1}(z) = \frac{z-2}{s^2 + s^{-2} -z}$ and
\begin{eqnarray*}
( S_m(z) - s^2 S_{m-1}(z) ) ( S_m(z) - s^{-2}S_{m-1}(z) ) 
&=& S^2_m(z) + S^2_{m-1}(z) - (s^2+s^{-2}) S_m(z) S_{m-1}(z) \\
&=&  \frac{s^2+s^{-2}-z}{z-2}.
\end{eqnarray*}
It follows that $\big( S_m(z) - S_{m-1}(z) \big)^2 \big( S_m(z) - s^2 S_{m-1}(z) \big) \big( S_m(z) - s^{-2}S_{m-1}(z) \big)=1$ and
$$w^2_{11} = \big( S_m(z) - S_{m-1}(z) \big)^2 \big( S_m(z) - s^{-2}S_{m-1}(z) \big)^2 = \frac{S_m(z) - s^{-2}S_{m-1}(z)}{S_m(z) - s^{2}S_{m-1}(z)}.$$
Note that $|S_m(z) - e^{-i\omega} S_{m-1}(z)| \ge |S_m(z) - e^{i\omega} S_{m-1}(z)|$ if and only if $\text{Im}( S_{m-1}(z)\overline{S_m(z)}) \ge 0$. Hence, for $\alpha \in (0, \alpha_{\CL_m})$, by the Schlafli formula we have 
$$\text{Vol} \, E_{\CL_m}(\alpha) 
= \int_{\alpha}^{\pi} 2\log \left| w_{11} \right| d\omega 
= \int_{\alpha}^{\pi} \log \left| \frac{S_m(z) - s^{-2} S_{m-1}(z)}{S_m(z) - s^2 S_{m-1}(z)} \right|  d\omega$$
where $s = e^{i\omega/2}$ and $z$, with $\text{Im} ( S_{m-1}(z)\overline{S_m(z)} ) \ge 0$, satisfy $$(s^2+s^{-2}+2-z) \big( S^2_m(z) + S^2_{m-1}(z) \big) - 2(s^2+s^{-2}) S_m(z) S_{m-1}(z) -z =0.$$
This completes the proof of Theorem \ref{thm:volume}.

\section{The A-polynomial 2-tuple of $\CL_m$} \label{sec:A-polynomial}

The canonical longitudes corresponding to the meridians $a$ and $b$ of $J(2m+1, 2n+1)$ are respectively
$\lambda_a = wa^{-2n}$ and $\lambda_b = \overline{w} b^{-2n}$, 
where $\overline{w} = (a^{-1}b)^m \big[ (ab^{-1})^m ab (a^{-1}b)^m \big]^n$ is the word obtained from $w$ by exchanging $a$ and $b$. 

Consider the canonical component $t=z$ of the character variety of $\CL_m = J(2m+1,2m+1)$. 
To compute the A-polynomial 2-tuple for this component, 
we first consider a representation $\rho: \pi_1(\CL_m) \to SL_2(\BC)$ of the form \eqref{rep} 
and find a polynomial relating $s_1$ and $w_{11}$ when both $t=z$ and $s^2_2 = (\overline{w}_{11})^2=1$ occur. 
Recall from Subsection \ref{nab} that $w_{11}$ and $\overline{w}_{11}$ are upper left entries of $\rho(w)$ and $\rho(\overline{w})$ respectively.

With $t=z$ and $s_2=1$, by equations \eqref{w11} and \eqref{ow11} we have 
\begin{eqnarray*}
w_{11} &=& - S_{m-1}(m) \big\{ 2 s_1^{-1}  S_m(z) - s_1^{-2} S_{m-1}(z) \big\} + S^2_m(z) \\
&=& ( S_m(z) - s_1^{-1}S_{m-1}(z) )^2
\end{eqnarray*}
and
\begin{eqnarray*}
\overline{w}_{11} &=& - S_{m-1}(z) \big\{ (s_1 + s_1^{-1} ) S_m(z) - S_{m-1}(z) \big\} +  S^2_m(z). \\
&=& ( S_m(z) - s_1 S_{m-1}(z) ) ( S_m(z) - s_1^{-1}S_{m-1}(z) ) .  
\end{eqnarray*}
Moreover, since $S^2_m(z) + S^2_{m-1}(z) = 1 + z S_m(z) S_{m-1}(z)$, the equation $t=z$ becomes  
\begin{eqnarray*}
0 &=& (2x-z) ( S^2_m(z) + S^2_{m-1}(z) ) - x^2 S_m(z) S_{m-1}(z) -z \\
 &=&  (2x-z) (1 + z S_m(z) S_{m-1}(z)) - x^2 S_m(z) S_{m-1}(z) -z \\
 &=& (x-z) \left( 2+(z-x)S_m(z) S_{m-1}(z)\right).
\end{eqnarray*}

Suppose $z-x=0$. Then $\overline{w}_{11} = - S_{m-1}(z) \big\{ z S_m(z) - S_{m-1}(z) \big\} +  S^2_m(z)=1$ and $$w_{11} = \big( S_m(x) - s^{-1}S_{m-1}(x) \big)^2=s^{2m}.$$
Here we use the fact that $S_j(s_1+s_1^{-1}) = (s_1^{j+1} - s_1^{-j-1})/(s_1-s_1^{-1})$ for all integers $j$.

Suppose $2+(z-x)S_m(z) S_{m-1}(z)=0$. This is equivalent to 
\begin{equation} \label{factor2}
\big( S_m(z) - s_1 S_{m-1}(z) \big) \big( S_m(z) - s_1^{-1}S_{m-1}(z) \big) = -1,
\end{equation}
since $S^2_m(z) + S^2_{m-1}(z) = 1 + z S_m(z) S_{m-1}(z)$. 
It follows that $\overline{w}_{11} = -1$ and $$w_{11} = \big( S_m(z) - s_1^{-1}S_{m-1}(z) \big)^2 = - \frac{S_m(z) - s_1^{-1}S_{m-1}(z)}{S_m(z) - s_1 S_{m-1}(z)}.$$  
Hence $S_m(z) = r S_{m-1}(z)$ where $r= \frac{s_1w_{11}+s_1^{-1}}{w_{11}+1}$. 
We have
$$1 = S^2_m(z) + S^2_{m-1}(z) - z S_m(z) S_{m-1}(z)  = S^2_{m-1}(z) (1 - z r + r^2),$$
which implies that $S^2_{m-1}(z) = (1 - z r + r^2)^{-1}.$ Equation \eqref{factor2} then becomes 
$$- 1 = S^2_{m-1}(z) (r - s_1)(r - s_1^{-1}) =  (r - s_1)(r - s_1^{-1})/(1 - z r + r^2).$$
By solving for $z$ from the above equation, we obtain
$$
z = 2 \left( r + \frac{1}{r} \right) - (s_1+s_1^{-1}) 
= 2 \left( \frac{s_1w_{11}+s_1^{-1}}{w_{11}+1} + \frac{w_{11}+1}{s_1w_{11}+s_1^{-1}} \right) - (s_1+s_1^{-1}).
$$
Now, by plugging this expression of $z$ into the equation $2+(z-x)S_m(z) S_{m-1}(z)=0$ we obtain a polynomial (depending on $m$) relating $s_1$ and $w_{11}$. Moreover, we can find a recurrence relation between these polynomials as follows. 

Let $P_m(x,z) = 2+(z-x)S_m(z) S_{m-1}(z)$. By Lemma \ref{chev3} we have $S_m(z) S_{m-1}(z)=(z^2-2)S_{m-1}(z) S_{m-2}(z)-S_{m-2}(z) S_{m-3}(z)+z$. This implies that
\begin{eqnarray*}
P_m &=& 2 + (z^2-2)(P_{m-1}-2)-(P_{m-2}-2)+z(z-x)\\
    &=& (z^2-2) P_{m-1} - P_{m-2} + 8 - z(z+x).
\end{eqnarray*}
Let $Q_m(s_1,w_{11})=s_1^2(w_{11}+1)^2(s_1^2w_{11}+1)^2 P_m(x,z)$. By replacing 
$$z=2 \left( \frac{s_1w_{11}+s_1^{-1}}{w_{11}+1} + \frac{w_{11}+1}{s_1w_{11}+s_1^{-1}} \right) - (s_1+s_1^{-1})$$ 
into the above recurrence relation for $P_m$ we have
$$Q_m = \alpha Q_{m-1} - Q_{m-2} +\beta$$ where
\begin{eqnarray*}
\alpha &=& (s_1^8+s_1^4) w_{11}^4+(-2 s_1^8+6 s_1^6+6 s_1^4-2 s_1^2) w_{11}^3 + (s_1^8-12 s_1^6+34 s_1^4-12 s_1^2+1) w_{11}^2\\
&& \qquad \qquad \quad \, \, \, + \, (-2 s_1^6+6 s_1^4+6 s_1^2-2) w_{11}+s_1^4+1,\\
\beta &=& -2(s_1^2-1)^2 \left( s_1^4 w_{11}^4 - (s_1^4+s_1^2) w_{11}^3 - 6 s_1^2 w_{11}^2 - (s_1^2+1)w_{11}+1\right).
\end{eqnarray*}

We have shown that $(\overline{w}_{11})^2=1$ and $(w_{11}-s_1^{2m})Q(s_1, w_{11})=0$ when both $t=z$ and $s_2=1$ occur. 
The same holds true when  both $t=z$ and $s_2=-1$ occur. 
This implies that $(w_{11}-s_1^{2m})Q(s_1, w_{11})=0$ when both $t=z$ and $s^2_2 = (\overline{w}_{11})^2=1$ occur. 

Similarly, we have $(\overline{w}_{11}-s_2^{2m})Q(s_2, \overline{w}_{11})=0$ when both $t=z$ and $s^2_1 = (w_{11})^2=1$ occur. 
Since the canonical longitudes corresponding to the meridians $a$ and $b$ of $\CL_m = J(2m+1, 2m+1)$ are respectively
$\lambda_a = wa^{-2m}$ and $\lambda_b = \overline{w} b^{-2m}$, we conclude that 
the A-polynomial 2-tuple corresponding to the canonical component of the character variety of $\CL_m$ is  
$[A(M,L), A(M,L)]$ where $A(M,L) = (L-1)Q_m(M, LM^{2m})$. 

This completes the proof of Theorem \ref{thm:A-polynomial}.

\section*{Acknowledgements} 
The author has been partially supported by a grant from the Simons Foundation (\#354595 to Anh Tran).


\end{document}